\documentclass[a4paper,11pt,reqno]{amsart}
\usepackage{amssymb,mathrsfs}

\newtheorem{theorem}{Theorem}
\newtheorem{proposition}[theorem]{Proposition}

\theoremstyle{remark}
\newtheorem{remark}[theorem]{Remark}
\newtheorem{example}[theorem]{Example}
\newtheorem*{acknowledgements}{Acknowledgements}

\newcommand{\bA}{\mathbf{A}}
\newcommand{\bN}{\mathbf{N}}
\newcommand{\bP}{\mathbf{P}}
\newcommand{\bR}{\mathbf{R}}
\newcommand{\cI}{\mathscr{I}}
\newcommand{\cO}{\mathscr{O}}
\newcommand{\clD}{\mathcal{D}}
\newcommand{\fa}{\mathfrak{a}}
\newcommand{\fb}{\mathfrak{b}}
\newcommand{\fm}{\mathfrak{m}}

\DeclareMathOperator{\codim}{codim}
\DeclareMathOperator{\mld}{mld}
\DeclareMathOperator{\ord}{ord}
\DeclareMathOperator{\Proj}{Proj}
\DeclareMathOperator{\Spec}{Spec}
\DeclareMathOperator{\wt}{wt}

\begin{document}
\title[Divisors computing the minimal log discrepancy]{Divisors computing the minimal log discrepancy on a smooth surface}

\author{Masayuki Kawakita}
\address{Research Institute for Mathematical Sciences, Kyoto University, Kyoto 606-8502, Japan}
\email{masayuki@kurims.kyoto-u.ac.jp}
\thanks{Partially supported by JSPS Grants-in-Aid for Young Scientists (A) 24684003 and for Scientific Research (C) 16K05099.}

\begin{abstract}
We study a divisor computing the minimal log discrepancy on a smooth surface. Such a divisor is obtained by a weighted blow-up. There exists an example of a pair such that any divisor computing the minimal log discrepancy computes no log canonical thresholds.
\end{abstract}

\maketitle

Let $P\in X$ be the germ of a smooth variety and $\fa$ be an $\bR$-ideal on $X$. The minimal log discrepancy $\mld_P(X,\fa)$ is an important invariant of singularities in view of Shokurov's reduction of the termination of flips \cite{S04}, but we do not have a good understanding of it. Presumably, one of its reasons is that it is unclear which divisor $E$ (over $X$) computes the minimal log discrepancy. That is, $E$ has centre $P$ and the log discrepancy of $E$ equals $\mld_P(X,\fa)$ or is negative. The purpose of this note is to study such $E$ when $X$ is a surface. We provide the description of $E$ in terms of a weighted blow-up.

\begin{theorem}\label{thm:surface}
Let $P\in X$ be the germ of a smooth surface and $\fa$ be an $\bR$-ideal on $X$.
\begin{enumerate}
\item\label{itm:surf-lc}
If $(X,\fa)$ is log canonical, then every divisor computing $\mld_P(X,\fa)$ is obtained by a weighted blow-up.
\item\label{itm:surf-nonlc}
If $(X,\fa)$ is not log canonical, then some divisor computing $\mld_P(X,
\linebreak
\fa)$ is obtained by a weighted blow-up.
\end{enumerate}
\end{theorem}

The log canonical threshold is another invariant of singularities, roughly corresponding to the minimal log discrepancy divided by the multiplicity. The log canonical threshold is considered to be easier to handle, because the minimal model program extracts a divisor computing minimal log discrepancy zero. For example, the ACC for log canonical thresholds is proved completely by Hacon, McKernan and Xu \cite{HMX14}. Hence in treating positive $\mld_P(X,\fa)$, it is a standard approach to find a suitable $\bR$-ideal $\fb$ such that $\mld_P(X,\fa\fb)$ equals zero. However even in dimension two, one can not expect the existence of the $\fb$ for which some divisor computes both $\mld_P(X,\fa)$ and $\mld_P(X,\fa\fb)$.

\begin{example}\label{exl:lct}
Let $P\in\bA^2=\Spec k[x_1,x_2]$ where $P$ is the origin. Consider the pair $(\bA^2,\fa^{2/3})$ for $\fa=(x_1^2+x_2^3,x_1x_2^2)$. There exists a unique divisor $E$ computing $\mld_P(\bA^2,\fa^{2/3})=2/3$, but $E$ does not compute $\mld_P(\bA^2,\fa^{2/3}\fb)$ for any $\bR$-ideal $\fb$ such that $\mld_P(\bA^2,\fa^{2/3}\fb)=0$.
\end{example}

\begin{remark}
I found Blum \cite{B16} independently proved that every divisor computing $\mld_P(X,\fa)\ge0$ on a smooth surface $X$ computes $\mld_P(X,\fb)=0$ for some $\fb$. This property is different from that discussed in Example \ref{exl:lct}, and follows from Theorem \ref{thm:surface} immediately. Indeed, let $E$ be the divisor obtained by the weighted blow-up of $P\in X$ with $\wt(x_1,x_2)=(w_1,w_2)$. Then $(X,\fb)$ for $\fb=(x_1^{w_2},x_2^{w_1})^{1/w_1+1/w_2}$ is lc since so is $(X,(x_1^{w_2})^{1/w_2}(x_2^{w_1})^{1/w_1})$. $E$ computes $\mld_P(X,\fb)=0$.
\end{remark}

We will fix the terminology before proceeding to the proofs of Theorem \ref{thm:surface} and Example \ref{exl:lct}. We work over an algebraically closed field $k$ of characteristic zero. The germ is considered at a closed point.

An $\bR$-\textit{ideal} on a variety $X$ is a formal product $\fa=\prod_j\fa_j^{r_j}$ of finitely many coherent ideal sheaves $\fa_j$ on $X$ with positive real exponents $r_j$. The \textit{order} of $\fa$ along a closed subvariety $C$ of $X$ is $\ord_C\fa=\sum_jr_j\ord_C\fa_j$, where $\ord_C\fa_j$ is the maximal $\nu\in\bN\cup\{+\infty\}$ satisfying $\fa_{j,\eta}\subset\cI_\eta^\nu$ for the ideal sheaf $\cI$ of $C$ and the generic point $\eta$ of $C$. The \textit{pull-back} of $\fa$ by a morphism $Z\to X$ is $\fa\cO_Z=\prod_j(\fa_j\cO_Z)^{r_j}$. If $Z\to X$ is birational and $Z$ is normal, then we set $\ord_E\fa=\ord_E\fa\cO_Z$ for a prime divisor $E$ on $Z$. If $Z\to X$ is a birational morphism from a smooth variety $Z$ whose exceptional locus is a divisor $\sum_iE_i$, then the \textit{weak transform} on $Z$ of $\fa$ is the $\bR$-ideal $\fa_Z=\prod_j(\fa_{jZ})^{r_j}$ defined by $\fa_{jZ}=\fa_j\cO_Z(\textstyle\sum_i(\ord_{E_i}\fa_j)E_i)$. This is different from the strict transform (see \cite[III Definition 5]{H64}).

A prime divisor $E$ on a normal variety $Y$ equipped with a birational morphism $Y\to X$ is called a divisor \textit{over} $X$, and the closure of the image in $X$ of $E$ is called the \textit{centre} of $E$ on $X$ and denoted by $c_X(E)$. We write $\clD_X$ for the set of all divisors over $X$. Two elements in $\clD_X$ are often identified if they define the same valuation on the function field of $X$.

Suppose $X$ to be smooth. The \textit{log discrepancy} of $E$ with respect to the pair $(X,\fa)$ is
\begin{align*}
a_E(X,\fa)=1+\ord_EK_{Y/X}-\ord_E\fa.
\end{align*}
We say that $(X,\fa)$ is \textit{log canonical} (\textit{lc}) if $a_E(X,\fa)\ge0$ for all $E\in\clD_X$, and is \textit{klt} (resp.\ \textit{plt}) if $a_E(X,\fa)>0$ for all $E\in\clD_X$ (resp.\ all $E\in\clD_X$ exceptional over $X$). The \textit{minimal log discrepancy} of $(X,\fa)$ at a closed point $P$ in $X$ is
\begin{align*}
\mld_P(X,\fa)=\inf\{a_E(X,\fa)\mid E\in\clD_X,\ c_X(E)=P\}.
\end{align*}
The $\mld_P(X,\fa)$ is either a non-negative real number or $-\infty$, and $(X,\fa)$ is lc about $P$ iff $\mld_P(X,\fa)\ge0$. If $E\in\clD_X$ satisfies that $c_X(E)=P$ and that $a_E(X,\fa)=\mld_P(X,\fa)$ (or is negative when $\mld_P(X,\fa)=-\infty$), then we say that $E$ \textit{computes} $\mld_P(X,\fa)$.

Let $P\in X$ be the germ of a smooth variety. Let $x_1,\ldots,x_c$ be a part of a regular system of parameters in $\cO_{X,P}$ and $w_1,\ldots,w_c$ be positive integers. For $w\in\bN$, let $\cI_w$ be the ideal in $\cO_X$ generated by all monomials $x_1^{s_1}\cdots x_c^{s_c}$ such that $\sum_{i=1}^cs_iw_i\ge w$. The \textit{weighted blow-up} of $X$ with $\wt(x_1,\ldots,x_c)=(w_1,\ldots,w_c)$ is $\Proj_X(\bigoplus_{w\in\bN}\cI_w)$. See \cite[6.38]{KSC04} for its explicit description.

\begin{proof}[Proof of Theorem \textup{\ref{thm:surface}}]
If $(X,\fa)$ is not lc, then any divisor computing $\mld_P(X,
\linebreak
\fa^t)=0$ for $t<1$ computes $\mld_P(X,\fa)=-\infty$. Thus (\ref{itm:surf-nonlc}) follows from (\ref{itm:surf-lc}).

We assume that $(X,\fa)$ is lc to show (\ref{itm:surf-lc}). Let $E$ be a divisor over $X$ which computes $\mld_P(X,\fa)$. For the maximal ideal $\fm$ in $\cO_{X,P}$, we set
\begin{align*}
w_1=\max_{x_1\in\fm\setminus\fm^2}\ord_Ex_1,\qquad w_2=\min_{x_2\in\fm\setminus\fm^2}\ord_Ex_2=\ord_E\fm.
\end{align*}
In order to verify the existence of the maximum $w_1$, let $Z\to X$ be a birational morphism from a smooth surface $Z$ on which $E$ appears as a divisor. Applying Zariski's subspace theorem \cite[(10.6)]{A98} to $\cO_{X,P}\subset\cO_{Z,Q}$ for a closed point $Q$ in $E$, one has an integer $w$ such that $\cO_Z(-wE)_Q\cap\cO_{X,P}\subset\fm^2$. Then $\ord_Ex_1<w$ for any $x_1\in\fm\setminus\fm^2$, so $w_1$ exists.

Take a regular system $x_1,x_2$ of parameters in $\cO_{X,P}$ such that $w_i=\ord_Ex_i$ for $i=1,2$. Let $Y$ be the weighted blow-up of $X$ with $\wt(x_1,x_2)=(w_1,w_2)$ and $F$ be its exceptional divisor. We write $w_i=gw'_i$ by the greatest common divisor $g$ of $w_1$ and $w_2$. We write $a_E=a_E(X,\fa)$ and $a_F=a_F(X,\fa)$ for simplicity. Then the inequality
\begin{align}\label{eqn:EF}
a_E\le a_F
\end{align}
holds because $E$ computes $\mld_P(X,\fa)$. It is enough to show that $E=F$.

For $i=1,2$, we let $C_i$ denote the strict transform in $Y$ of the curve defined on $X$ by $(x_i)$. Then one computes that
\begin{align*}
w_i=\ord_Ex_i=\ord_Fx_i\cdot\ord_EF+\ord_EC_i=w'_i\ord_EF+\ord_EC_i.
\end{align*}
The $\ord_EC_i$ is positive iff the centre $c_Y(E)$ lies on $C_i$. Since $C_1$ and $C_2$ are disjoint, at least one of $\ord_EC_1$ and $\ord_EC_2$ is zero. By
\begin{align*}
\frac{w'_1}{w'_2}=\frac{w_1}{w_2}=\frac{w'_1\ord_EF+\ord_EC_1}{w'_2\ord_EF+\ord_EC_2},
\end{align*}
one concludes that $\ord_EC_1=\ord_EC_2=0$. Hence if $E\neq F$, then the centre $c_Y(E)$ must be a closed point $Q$ in $F\setminus(C_1+C_2)$.

Assuming that $c_Y(E)=Q\in F\setminus(C_1+C_2)$, we will derive a contradiction. One has the weak transform $\fa_Y$ of $\fa$ on the germ $Q\in Y$. Then
\begin{align*}
a_E(Y,F,\fa_Y)&=a_E(Y,(1-a_F)F,\fa_Y)-a_F\ord_EF=a_E-a_F\ord_EF\le0
\end{align*}
by (\ref{eqn:EF}), so $(Y,F,\fa_Y)$ is not plt about $Q$. Then $(F,\fa_Y\cO_F)$ is not klt about $Q$ by inversion of adjunction \cite[Sect.\ 4.1]{Ko13}. This means that
\begin{align}\label{eqn:ge1}
\ord_Q(\fa_Y\cO_F)\ge1.
\end{align}

Write $\fa=\prod_j\fa_j^{r_j}$. We take a general member $f_j$ in $\fa_j$ and express
\begin{align*}
f_j=c_jx_1^{s_{1j}}x_2^{s_{2j}}\prod_{\lambda\in k^\times}(x_1^{w'_2}+\lambda x_2^{w'_1})^{t_{\lambda j}}+h_j
\end{align*}
with $s_{ij},t_{\lambda j}\in\bN$, $c_j\in k^\times$ and $\ord_Fh_j>\ord_Ff_j$, so that
\begin{align}\label{eqn:ordF}
\ord_F\fa_j=\ord_Ff_j=s_{1j}w'_1+s_{2j}w'_2+\sum_{\lambda\in k^\times}t_{\lambda j}w'_1w'_2.
\end{align}
Then by the identification of $F$ with the weighted projective space $\bP(w'_1,w'_2)$, the general member in $\fa_Y\cO_F$ is
\begin{align*}
\prod_j\Bigl(x_1^{s_{1j}}x_2^{s_{2j}}\prod_{\lambda\in k^\times}(x_1^{w'_2}+\lambda x_2^{w'_1})^{t_{\lambda j}}\Bigr)^{r_j}.
\end{align*}

Let $\lambda_0\in k^\times$ be the unique unit such that $Q$ lies on the strict transform $D$ in $Y$ of the curve defined on $X$ by $(x_1^{w'_2}+\lambda_0x_2^{w'_1})$. Then
\begin{align*}
\ord_Q(\fa_Y\cO_F)&=\ord_Q\prod_j\Bigl(x_1^{s_{1j}}x_2^{s_{2j}}\prod_{\lambda\in k^\times}(x_1^{w'_2}+\lambda x_2^{w'_1})^{t_{\lambda j}}\Bigr)^{r_j}\\
&=\sum_jr_jt_{\lambda_0j}\le\sum_jr_j(w'_1w'_2)^{-1}\ord_F\fa_j=(w'_1w'_2)^{-1}\ord_F\fa,
\end{align*}
where the inequality follows from (\ref{eqn:ordF}). Since $0\le a_F=a_F(X)-\ord_F\fa=w'_1+w'_2-\ord_F\fa$, one has that $\ord_F\fa\le w'_1+w'_2$. Thus,
\begin{align*}
\ord_Q(\fa_Y\cO_F)\le(w'_1w'_2)^{-1}(w'_1+w'_2)=1/w'_1+1/w'_2.
\end{align*}
Combining (\ref{eqn:ge1}) and this, we obtain that $1\le1/w'_1+1/w'_2$, whence $w'_2=1$ by the coprimeness of $w'_1$, $w'_2$. But then
\begin{align*}
\ord_E(x_1+\lambda_0 x_2^{w'_1})=w'_1\ord_EF+\ord_ED>w'_1\ord_EF=w_1,
\end{align*}
which contradicts the definition of $w_1$.
\end{proof}

\begin{remark}
It is proved in \cite[Theorem 6.40]{KSC04} after Var\v{c}enko that if $P\in X$ is the germ of a smooth complex analytic surface and $(X,tC)$ is lc but not klt for a curve $C$ on $X$, then some divisor $E$ satisfying $a_E(X,tC)=0$ is obtained by a weighted blow-up ($E$ may be a curve on $X$).
\end{remark}

\begin{proof}[Proof of Example \textup{\ref{exl:lct}}]
Let $C$ be the curve on $\bA^2$ defined by $(x_1^2+x_2^3)$. Let $X_1$ be the blow-up of $X$ at $P$ and $E_1$ be its exceptional divisor. For $i=2,3,4$, let $X_i$ be the blow-up of $X_{i-1}$ at $E_{i-1}\cap C_{i-1}$ for the strict transform $C_{i-1}$ of $C$, and $E_i$ be its exceptional divisor. Then
\begin{align*}
\fa\cO_{X_4}=\cO_{X_4}(-2E_1-3E_2-6E_3-7E_4)
\end{align*}
by the same notation $E_i$ for its strict transform, and $X_4$ is a log resolution of $(\bA^2,\fa)$. One computes $a_{E_1}(\bA^2,\fa^{2/3})=2/3$, $a_{E_2}(\bA^2,\fa^{2/3})=1$, $a_{E_3}(\bA^2,\fa^{2/3})=1$ and $a_{E_4}(\bA^2,\fa^{2/3})=4/3$. Thus $\mld_P(\bA^2,\fa^{2/3})=2/3$ and it is computed only by $E_1$.

Suppose that $\mld_P(\bA^2,\fa^{2/3}\fb)=0$ for an $\bR$-ideal $\fb=\prod_j\fb_j^{r_j}$ and that it is computed by $E_1$. Then,
\begin{align*}
\ord_{E_1}\fb=a_{E_1}(\bA^2,\fa^{2/3})=2/3.
\end{align*}
On the other hand, since $\fb_j\subset\fm^{\ord_{E_1}\fb_j}$ for the maximal ideal $\fm$ defining $P$, one has that $\ord_{E_3}\fb_j\ge\ord_{E_1}\fb_j\cdot\ord_{E_3}\fm=2\ord_{E_1}\fb_j$, so
\begin{align*}
\ord_{E_3}\fb=\sum_jr_j\ord_{E_3}\fb_j\ge2\sum_jr_j\ord_{E_1}\fb_j=2\ord_{E_1}\fb=4/3.
\end{align*}
But then
\begin{align*}
a_{E_3}(\bA^2,\fa^{2/3}\fb)=a_{E_3}(\bA^2,\fa^{2/3})-\ord_{E_3}\fb\le-1/3,
\end{align*}
which contradicts $\mld_P(\bA^2,\fa^{2/3}\fb)=0$.
\end{proof}

The description in Theorem \ref{thm:surface} holds only in dimension two.

\begin{example}
Let $P\in\bA^3=\Spec k[x_1,x_2,x_3]$ where $P$ is the origin. Consider the pair $(\bA^3,\fa^{4/3})$ for $\fa=(x_1x_2+x_3^2)+(x_1,x_2,x_3)^3$, which has appeared substantially in \cite[Exercise 6.45]{KSC04}. There exists a unique divisor $E$ computing $\mld_P(\bA^3,\fa^{4/3})=0$, but $E$ is not obtained by a weighted blow-up.

Indeed, let $X_1$ be the blow-up of $\bA^3$ at $P$ and $E_1$ be its exceptional divisor. Let $S_1$ be the strict transform in $X_1$ of the surface defined on $\bA^3$ by $(x_1x_2+x_3^2)$. Let $X_2$ be the blow-up of $X_1$ along the curve $S_1\cap E_1$ and $E_2$ be its exceptional divisor. Then
\begin{align*}
\fa\cO_{X_2}=\cO_{X_2}(-2E_1-3E_2)
\end{align*}
by the same notation $E_1$ for its strict transform, and $X_2$ is a log resolution of $(\bA^3,\fa)$. One computes $a_{E_1}(\bA^3,\fa^{4/3})=1/3$ and $a_{E_2}(\bA^3,\fa^{4/3})=0$. Thus $\mld_P(\bA^3,\fa^{4/3})=0$ and it is computed only by $E_2$. However, $E_2$ is not obtained by a weighted blow-up.
\end{example}

We remark a supplement to Theorem \ref{thm:surface} which holds in any dimension.

\begin{proposition}
Let $P\in X$ be the germ of a smooth variety and $\fa$ be an $\bR$-ideal on $X$. Let $E$ be the divisor obtained by the blow-up of $X$ at $P$.
\begin{enumerate}
\item\label{itm:supple-ge}
If $\ord_P\fa\le1$, then $E$ computes $\mld_P(X,\fa)$.
\item\label{itm:supple->}
If $\ord_P\fa<1$, then $E$ is the unique divisor computing $\mld_P(X,\fa)$.
\end{enumerate}
\end{proposition}

\begin{proof}
The $a_E(X,\fa)$ is the limit of $a_E(X,\fa^{1-\epsilon})$ and $\mld_P(X,\fa)$ is that of $\mld_P(X,\fa^{1-\epsilon})$ when $\epsilon$ goes to zero from above. Thus (\ref{itm:supple-ge}) follows from (\ref{itm:supple->}).

Assume $\ord_P\fa<1$ to show (\ref{itm:supple->}). Let $F$ be an arbitrary divisor over $X$ other than $E$ such that $c_X(F)=P$. Setting $(X_0,C_0)=(X,P)$, we build a tower of finitely many birational morphisms
\begin{align*}
X_n\to\cdots\to X_1\to X_0=X
\end{align*}
such that
\begin{enumerate}
\item[(a)]
$X_i\to X_{i-1}$ is the composition of an open immersion $X_i\to Y_i$ and the blow-up $Y_i\to X_{i-1}$ along $C_{i-1}$,
\item[(b)]
$C_i=c_{X_i}(F)$ is smooth and non-empty, and
\item[(c)]
$C_n$ is a divisor ($n$ is taken smallest).
\end{enumerate}
Note that $n\ge2$ since $E\neq F$. Write $\fa_i$ for the weak transform on $X_i$ of $\fa$ and $E_i$ for the exceptional divisor of $X_i\to X_{i-1}$.

Fixing a subvariety $D_i$ of $C_i$ such that $D_i\to C_{i-1}$ is quasi-finite and dominant, by \cite[III Lemmata 7 and 8]{H64} one has that
\begin{align*}
\ord_{C_i}\fa_i\le\ord_{D_i}\fa_i\le\ord_{C_{i-1}}\fa_{i-1}.
\end{align*}
Thus $\ord_{C_i}\fa_i<1$ for any $i$ by our assumption $\ord_{C_0}\fa_0<1$. On the other hand, one computes that
\begin{align*}
a_{E_i}(X,\fa)=\sum_{j\in I_{i-1}}(a_{E_j}(X,\fa)-1)+\codim_{X_{i-1}}C_{i-1}-\ord_{C_{i-1}}\fa_{i-1},
\end{align*}
where $I_{i-1}$ denotes the set of $1\le j\le i-1$ such that $C_{i-1}$ lies on the strict transform of $E_j$. Using $\ord_{C_{i-1}}\fa_{i-1}<1$, one concludes inductively that
\begin{align*}
a_{E_i}(X,\fa)>1,\qquad a_{E_i}(X,\fa)>a_{E_{i-1}}(X,\fa)
\end{align*}
for any $i$. In particular,
\begin{align*}
a_F(X,\fa)=a_{E_n}(X,\fa)>a_{E_1}(X,\fa)=a_E(X,\fa),
\end{align*}
and (\ref{itm:supple->}) follows.
\end{proof}

\begin{acknowledgements}
Part of the research was achieved during my visits at Korea Institute for Advanced Study and at National Taiwan University. I should like to thank Dr.\ Z. Zhu and Professor J. A. Chen for their warm hospitalities. I am also grateful to Professors S. Helmke and J. Koll\'ar for comments.
\end{acknowledgements}


\begin{thebibliography}{9}
\bibitem{A98}
S. S. Abhyankar,
\textit{Resolution of singularities of embedded algebraic surfaces}, 2nd, enl.\ ed.,
Springer Monographs in Mathematics, Springer (1998)
\bibitem{B16}
H. Blum,
On divisors computing MLD's and LCT's,
arXiv:1605.09662
\bibitem{HMX14}
C. D. Hacon, J. McKernan and C. Xu,
ACC for log canonical thresholds,
Ann.\ Math.\ (2) \textbf{180}, No.\ 2, 523-571 (2014)
\bibitem{H64}
H. Hironaka,
Resolution of singularities of an algebraic variety over a field of characteristic zero. I, II,
Ann.\ Math.\ (2) \textbf{79}, 109-203, 205-326 (1964)
\bibitem{Ko13}
J. Koll\'ar,
\textit{Singularities of the minimal model program},
Cambridge Tracts in Mathematics \textbf{200}, Cambridge University Press (2013)
\bibitem{KSC04}
J. Koll\'ar, K. Smith and A. Corti,
\textit{Rational and nearly rational varieties},
Cambridge Studies in Advanced Mathematics \textbf{92}, Cambridge University Press (2004)
\bibitem{S04}
V. V. Shokurov,
Letters of a bi-rationalist. V: Minimal log discrepancy and termination of log flips,
Proc.\ Steklov Inst.\ Math.\ \textbf{246}, 315-336 (2004); translation from Tr.\ Mat.\ Inst.\ Steklova \textbf{246}, 328-351 (2004)
\end{thebibliography}
\end{document}